\documentclass[a4paper]{article}
\usepackage{amsthm,amssymb,amsmath,enumerate}
\usepackage{algorithm}

\usepackage{tikz}
\usetikzlibrary{backgrounds,scopes}

\usepackage{pifont}% http://ctan.org/pkg/pifont
\newcommand{\cmark}{\ding{51}}%
\newcommand{\xmark}{\ding{55}}%

\newtheorem{definition}{Definition}
\newtheorem{proposition}[definition]{Proposition}
\newtheorem{theorem}[definition]{Theorem}
\newtheorem{corollary}[definition]{Corollary}
\newtheorem{lemma}[definition]{Lemma}
\newtheorem{conjecture}[definition]{Conjecture}

\newcommand{\charf}[1]{\raisebox{\depth}{\(\chi\)}_{#1}}
\newcommand{\rstr}[1]{|_{#1}}
\newcommand{\og}{\textrm{\it og}}

\newcommand{\1}{{\rm 1\hspace*{-0.4ex}%
\rule{0.1ex}{1.52ex}\hspace*{0.2ex}}}

\newcommand{\comment}[1]{}

\newcommand{\R}{\mathbb R}

\newcommand{\ssp}{\textrm{\rm SSP}}
\newcommand{\tstab}{\textrm{\rm TSTAB}}
\newcommand{\aweb}[2]{\ensuremath{A^{#2}_{#1}}}

\newcommand{\emtext}[1]{\text{\em #1}}

\newcommand{\sm}{\setminus}

\def\drawfive{
\foreach \i in {0,1,2,3,4}{
    \draw[hedge] (72*\i+90:\radius) -- (72*\i+72+90:\radius); 
}
\foreach \i in {0,1,2,3,4}{
    \node[hvertex] (v\i) at (72*\i+90:\radius){}; 
}
}

\def\drawfivehole{
\foreach \i in {0,1,2,3,4}{
    \draw[hedge] (\angle*\i+90:\radius) -- (\angle*\i+\angle+90:\radius); 
}
\foreach \i in {0,1,2,3,4}{
    \node[hvertex] (v\i) at (\angle*\i+90:\radius){}; 
}
\node[hvertex] (c) at (0,0){};
}

\tikzstyle{hvertex}=[thick,circle,inner sep=0.cm, minimum size=2.5mm, fill=white, draw=black]
\tikzstyle{hedge}=[very thick]

\def\converti#1{\pgfmathparse{int(mod(#1*(\numberk+1),\numbern))}\pgfmathresult}

\def\nodelabels{}
\def\labeldist{12pt}

\newcommand{\antiweb}[2]{
\def\numbern{#1}
\pgfmathtruncatemacro{\numbernminusone}{\numbern-1}
\def\numberk{#2}
\def\angle{360/\numbern}

\pgfmathtruncatemacro{\firstnbh}{\numberk+1}
\pgfmathtruncatemacro{\lastnbh}{floor(\numbern/2)}

\foreach \i in {0,1,...,\numbernminusone}{  
  \pgfmathtruncatemacro{\myresult}{mod(\i*(\numberk+1),\numbern)}
  \node[hvertex] (v\myresult) at (90+\i*\angle:\radius){};
  \node at (90+\i*\angle:\labeldist+\radius){\nodelabels};
}

\foreach \i in {0,1,...,\numbernminusone}{  
  \foreach \j in {\firstnbh,...,\lastnbh}{
    \pgfmathtruncatemacro{\othervx}{mod(\i+\j,\numbern)}
    \draw[hedge] (v\i) -- (v\othervx);
  }
}
}

\title{$t$-perfection in $P_5$-free graphs}
\author{Henning Bruhn and Elke Fuchs}
\date{}

\begin{document}
\maketitle

\begin{abstract}
A graph is called $t$-perfect if its stable set polytope 
is fully described by non-negativity, edge and odd-cycle constraints.
We characterise $P_5$-free $t$-perfect  graphs
in terms of forbidden $t$-minors. Moreover, we show that 
$P_5$-free $t$-perfect  graphs can always be coloured with three colours, 
and that they can be recognised in polynomial time. 
\end{abstract}

\section{Introduction}

There are three quite different views on \emph{perfect graphs}, a view in
terms of colouring, a polyhedral and a structural view. 
Perfect graphs can be seen as: 
\begin{itemize}
\item the graphs for which the chromatic number $\chi(H)$ 
always equals the clique number $\omega(H)$, and that in any induced subgraph~$H$;
\item the graphs for which the \emph{stable set polytope}, the convex
hull of stable sets, is fully described by non-negativity and clique constraints; and
\item the graphs that do not contain any \emph{odd hole} (an induced cycle of odd length
at least~$5$) or their complements, \emph{odd antiholes}.
\end{itemize}
(The polyhedral characterisation is due to Fulkerson~\cite{Fulkerson72} and Chv\'atal~\cite{Chvatal75},
while the third item, the strong perfect graph theorem, 
was proved by Chudnovsky, Robertson, Seymour and Thomas~\cite{SPGT}.)

In this article, we work towards a similar threefold view on \emph{$t$-perfect}
graphs. 
These are graphs that, similar to perfect graphs, have a particularly simple
stable set polytope. For a graph to be $t$-perfect its stable set polytope
needs to be given by non-negativity, edge and odd-cycle constraints; for 
precise definitions we defer to the next section.
The concept of $t$-perfection, due to 
Chv\'atal~\cite{Chvatal75}, thus takes its motivation from the 
polyhedral aspect of perfect graphs. The corresponding colouring and
structural view, however, is still missing. For some graph classes, 
though, claw-free graphs for instance~\cite{tperfect}, 
the list of minimal obstructions for $t$-perfection is known. 
We extend this list to  $P_5$-free graphs.
(A graph is $P_5$-free if it does not contain the path on five vertices
as an induced subgraph.)

Perfection is preserved under vertex deletion, and the same 
is true for $t$-perfection. 
There is 
a second simple operation that maintains $t$-perfection:
a \emph{$t$-contraction}, which is only allowed at a vertex with stable neighbourhood,
contracts all the incident edges. 
Any graph obtained by a sequence of vertex deletions and $t$-contractions
is a \emph{$t$-minor}. 
The concept of $t$-minors makes it more convenient to characterise $t$-perfection 
in certain graph classes as it allows for more succinct lists
of obstructions. 

For that characterisation denote by  $C^k_n$ 
the $k$th power of the $n$-cycle $C_n$, that is, 
the the graph obtained from $C_n$ by adding an edge 
between any two vertices of distance at most~$k$ in $C_n$.
We, moreover, write $\overline G$ for the complement of a graph $G$, 
and $K_n$ for the complete graph on $n$ 
vertices and $W_n$ for the wheel with $n+1$ vertices.

\begin{theorem}\label{thmmainres}
Let $G$ be a $P_5$-free graph. 
Then $G$ is $t$-perfect if and only if it does not contain 
any of $K_4$, $W_5$, $C^2_7$, $\overline{C_{10}^{2}}$ or $\overline{C_{13}^{3}}$ as a $t$-minor. 
\end{theorem}
%Note that a similar result for $P_4$-free graphs is easy to obtain: Indeed,
%$P_4$-free graphs are perfect. Thus a $P_4$-free graph is $t$-perfect if and only if 
%it does not contain a $K_4$.
This answers a question of  Benchetrit~\cite[p.~76]{YohPhD}.

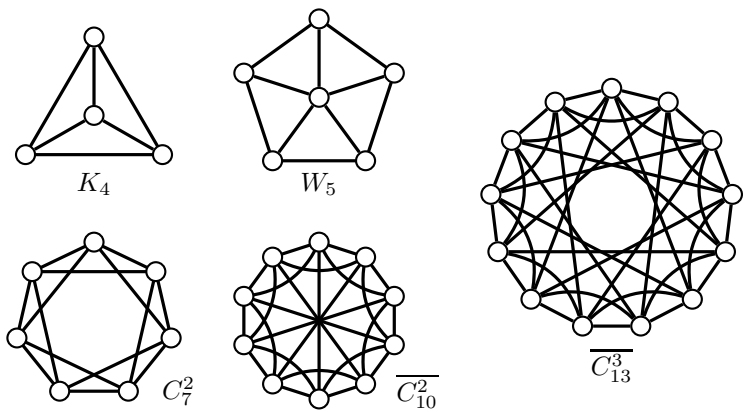
\begin{figure}[ht]
\centering
\begin{tikzpicture}[scale=0.8,auto]

\def\radius{1.3cm}
\def\hshift{3.7cm}
\def\vshift{-3.7cm}

\begin{scope}[shift={(0,-0.3)}]
\node[hvertex] (c) at (0,0){};
\foreach \i in {0,1,2}{
  \begin{scope}[on background layer] 
    \draw[hedge] (90+120*\i:\radius) -- (210+120*\i:\radius);
  \end{scope}
  \node[hvertex] (v\i) at (90+120*\i:\radius){};
  \draw[hedge] (c) -- (v\i);
}
\end{scope} 
\node at (0,-1.1*\radius){$K_4$};

\begin{scope}[shift={(\hshift,0)}]
\node[hvertex] (c) at (0,0){};
\foreach \i in {0,1,2,3,4}{
  \begin{scope}[on background layer] 
    \draw[hedge] (90+72*\i:\radius) -- (162+72*\i:\radius);
  \end{scope}
  \node[hvertex] (v\i) at (90+72*\i:\radius){};
  \draw[hedge] (c) -- (v\i);
}
\node at (0,-1.1*\radius){$W_5$};
\end{scope} 

\begin{scope}[shift={(0,\vshift)}]
\def\angle{360/7}
\foreach \i in {0,1,2,3,4,5,6}{
  \begin{scope}[on background layer] 
    \draw[hedge] (90+\angle*\i:\radius) -- (90+\angle+\angle*\i:\radius);
    \draw[hedge] (90+\angle*\i:\radius) -- (90+2*\angle+\angle*\i:\radius);
  \end{scope}
  \node[hvertex] (v\i) at (90+\angle*\i:\radius){};
}
\node at (320:1.4*\radius){$C^2_7$};
\end{scope} 

\begin{scope}[shift={(\hshift,\vshift)}]
\def\angle{360/10}
\foreach \i in {0,1,2,3,4}{
  \draw[hedge] (90+\angle*\i:\radius) -- (90+5*\angle+\angle*\i:\radius);
}
\foreach \i in {0,1,2,3,4,5,6,7,8,9}{
  \begin{scope}[on background layer] 
    \draw[hedge] (90+\angle*\i:\radius) -- (90+\angle+\angle*\i:\radius);
    \draw[hedge,bend left=30] (90+\angle*\i:\radius) to (90+2*\angle+\angle*\i:\radius);
  \end{scope}
  \node[hvertex] (v\i) at (90+\angle*\i:\radius){};
}
\node at (320:1.4*\radius){$\quad\overline{C^2_{10}}$};
\end{scope} 

\begin{scope}[shift={(2.3*\hshift,0.5*\vshift)}]
\def\radius{2cm}

\def\angle{360/13}

\foreach \i in {0,1,...,12}{  
  \pgfmathtruncatemacro{\myresult}{mod(\i*4,13)}
  \node[hvertex] (v\myresult) at (90+\i*\angle:\radius){};
}

\foreach \i in {0,1,...,12}{  
  \pgfmathtruncatemacro{\othervx}{mod(\i+4,13)}
  \draw[hedge] (v\i) -- (v\othervx);
  \pgfmathtruncatemacro{\othervx}{mod(\i+5,13)}
  \draw[hedge, bend right=30] (v\i) to (v\othervx);
  \pgfmathtruncatemacro{\othervx}{mod(\i+6,13)}
  \draw[hedge] (v\i) -- (v\othervx);
}
\node at (0,-1.3*\radius){$\overline{C^3_{13}}$};

\end{scope}

\end{tikzpicture}

\caption{Forbidden $t$-minors in $P_5$-free graphs}
\label{culpritsfig}
\end{figure}

The forbidden graphs of the theorem are \emph{minimally $t$-imperfect}, in the sense that
they are $t$-imperfect but any of their proper $t$-minors are $t$-perfect.
Odd wheels, even M\"obius ladders (see Section~\ref{sec:nb}), the cycle power $C^2_7$ 
and the graph $\overline{C_{10}^{2}}$ are known to be minimally $t$-imperfect.
The graph $\overline{C_{13}^{3}}$ appears here for the first time as a minimally $t$-imperfect graph.
We prove this in Section~\ref{secMinAweb}, where we also present two more minimally $t$-imperfect graphs.

A starting point for Theorem~\ref{thmmainres}
was the observation of Benchetrit~\cite[p.~75]{YohPhD} 
that   $t$-minors of $P_5$-free graphs are again $P_5$-free. Thus, 
any occurring minimally $t$-imperfect graph will be $P_5$-free, too.
This helped to whittle down the list of prospective forbidden $t$-minors.
We prove Theorem~\ref{thmmainres} in Sections~\ref{sectionharm} and~\ref{secmainres}. %BLUB

A graph class in which $t$-perfection is quite well understood is the class of 
\emph{near-bipartite} graphs; these are the graphs that become bipartite whenever the neighbourhood
of any vertex is deleted. In the course of the proof of Theorem~\ref{thmmainres} we
make use of results of Shepherd~\cite{Shepherd95} and of Holm, Torres and Wagler~\cite{HTW10}:
together they yield a description of $t$-perfect near-bipartite graphs 
in terms of forbidden induced subgraphs. We discuss this in Section~\ref{sec:nb}.

As  a by-product of
the proof of Theorem~\ref{thmmainres} 
we also obtain a polynomial-time 
algorithm to check for $t$-perfection in $P_5$-free graphs (Theorem~\ref{polythm}). 

Finally, in Section~\ref{colsec}, we turn to the third defining aspect of
perfect graphs: colouring. 
Shepherd and Seb\H o conjectured that every $t$-perfect graph can be coloured with four colours, which would be tight. 
For $t$-perfect $P_5$-graphs 
we show (Theorem~\ref{P5colthm}) that already three colours suffice.
We, furthermore, offer a conjecture that would, if true, characterise 
$t$-perfect graphs in terms of (fractional) colouring, in a way that is quite
similar as for perfect graphs.

\medskip

We end the introduction with a brief discussion of the literature
on $t$-perfect graphs. A general treatment may be found 
in Gr\"otschel, Lov\'asz and Schrijver~\cite[Ch.~9.1]{GLS88}
as well as in Schrijver~\cite[Ch.~68]{LexBible}. The most comprehensive source
of literature references is surely the PhD thesis of Benchetrit~\cite{YohPhD}. 
A part of the literature is devoted to proving $t$-perfection for certain 
graph classes. For instance, 
Boulala and Uhry~\cite{BouUhr79} established the $t$-perfection of series-parallel graphs.
Gerards~\cite{Gerards89} extended this to  graphs that do not contain an 
\emph{odd-$K_4$} as a subgraph
(an odd-$K_4$ is a subdivision of $K_4$ in which every triangle
becomes an odd circuit).
Gerards and Shepherd~\cite{GS98} characterised the graphs with all subgraphs $t$-perfect,
while Barahona and Mahjoub~\cite{BM94} described the $t$-imperfect subdivisions of $K_4$.
Wagler~\cite{Wagler04} gave a complete description of the stable set polytope of antiwebs, the complements
of cycle powers. These are near-bipartite graphs that also play a prominent role in the proof 
of Theorem~\ref{thmmainres}.
See also  Wagler~\cite{Wagler2005} for an extension to a more general class of near-bipartite graphs. 
The complements of near-bipartite graphs are the quasi-line graphs. 
Chudnovsky and Seymour~\cite{ChuSey05}, 
and Eisenbrand, Oriolo, Stauffer and Ventura~\cite{Eisenbrand05} 
determined the precise structure of the stable set
polytope of quasi-line graphs. Previously, this was a conjecture of Ben Rebea~\cite{benrebea81}.

Algorithmic aspects of $t$-perfection were also studied: Gr\"otschel, Lov\'asz and Schrijver~\cite{GLS86}
showed that the max-weight stable set problem can be solved in polynomial-time in $t$-perfect graphs.
Eisenbrand et al.~\cite{EFGK03} found a combinatorial algorithm for the unweighted case.

\section{Definitions}
All the graphs  in this article are finite, simple and do not have parallel edges or loops. 
In general, we follow the notation of Diestel~\cite{Diestel}, where also any 
missing elementary facts about graphs may be found.

Let $G=(V,E)$ be a graph.
The \emph{stable set polytope} $\ssp(G)\subseteq\mathbb R^{V}$ of $G$ 
is defined as the convex 
hull of the characteristic vectors of stable, i.e.\ independent, subsets of $V$.
The characteristic vector of a subset $S$ of the set $V$ is the vector $\charf{S}\in \{0,1\}^{V}$ with $\charf{S}(v)=1$ if $v\in S$ and $0$ otherwise. 
We define a second polytope $\tstab(G)\subseteq\mathbb R^V$ for $G$, given by
\begin{eqnarray*}
&&x\geq 0,\notag\\
&&\label{tstab}x_u+x_v\leq 1\text{ for every edge }uv\in E,\\ 
&&\sum_{v\in V(C)}x_v\leq \left\lfloor\frac{ |C|}{2}\right\rfloor\text{ for every induced odd cycle }C
\text{ in }G.\notag
\end{eqnarray*}
These inequalities are respectively known as non-negativity, edge and
odd-cycle inequalities. Clearly, $\ssp(G)\subseteq \tstab(G)$.

Then, the graph $G$ is called \emph{$t$-perfect} if $\ssp(G)$ and
$\tstab(G)$ coincide. Equivalently, $G$ is $t$-perfect if and only if
$\tstab(G)$ is an integral polytope, i.e.\ if all its vertices
are integral vectors. It is easy to see that bipartite graphs are $t$-perfect.
The smallest \emph{$t$-imperfect} graph is $K_4$. Indeed, the vector $\tfrac{1}{3}\1$
lies in $\tstab(K_4)$ but not in $\ssp(K_4)$.

It is easy to verify that  vertex deletion preserves $t$-perfection.
Another operation that keeps $t$-perfection was found by 
Gerards and Shepherd~\cite{GS98}: whenever there is a vertex $v$, so that its
neighbourhood is stable, we may contract all edges incident with $v$
simultaneously. We will call this operation a  \emph{$t$-contraction at~$v$}.
Any graph that is  obtained from $G$ by a sequence of vertex deletions 
and $t$-contractions is a  \emph{$t$-minor of $G$}.
Let us point out that any $t$-minor of a $t$-perfect graph is again $t$-perfect.
%\medskip

%The complete graph on $n$ vertices is denoted by $K_n$ and the wheel on $n+1$ vertices by $W_n$. A cycle on $n$ vertices is abbreviated to $C_n$ and a path on $n$ vertices to $P_n$. We say that a graph is $P_k$-free, if it does not contain an induced path on $k$ vertices.

\section{$t$-perfection in near-bipartite graphs}\label{sec:nb}

Part of the proof of Theorem~\ref{thmmainres} consists in a reduction to 
\emph{near-bipartite} graphs. A  graph is near-bipartite if it becomes bipartite
 whenever 
the neighbourhood of any of its vertices is deleted.
We will need a characterisation of $t$-perfect near-bipartite graphs in terms 
of forbidden induced subgraphs. Fortunately, such a characterisation follows 
immediately from results of Shepherd~\cite{Shepherd95} and of Holm, Torres and Wagler~\cite{HTW10}.

We need a bit of notation.
Examples of near-bipartite graphs are antiwebs:
an \emph{antiweb} $\overline{C_{n}^{k}}$ is the complement
of the $k$th power of the $n$-cycle $C_n$. The antiweb is 
\emph{prime} if $n \geq 2k+2$ and $k+1$, $n$ are relatively prime.  
We simplify the notation for  
antiwebs $\overline{C^k_n}$ slightly 
by writing $\aweb{n}{k}$ instead. 
\emph{Even M\"obius ladders}, the graphs $\aweb{4t+4}{2t}$, are  prime antiwebs; see  Figure~\ref{moebius8fig}
for the M\"obius ladder $\overline{C^2_8}$. We view $K_4$ alternatively as the smallest odd wheel $W_3$ or as 
the smallest even M\"obius ladder $\overline{C^0_4}$.
Trotter~\cite{Trotter75} found that prime antiwebs  give rise to facets in the stable set polytope---we
only need that prime antiwebs other than odd cycles are $t$-imperfect, a 
fact that is easier to check. 

\begin{figure}[ht]
\centering
\begin{tikzpicture}
\def\radius{1.5}
\tikzstyle{oben}=[line width=2pt,double distance=1.2pt,draw=white,double=black]
  \def\numbert{4}
  \def\angle{360/\numbert}
  \def\startangle{90+\angle/2}
  \def\endangle{90-\angle/2}
  \def\mxshift{0.6}
  \def\myshift{-2}

\begin{scope}[y={(0.3cm,0.6cm)}]
  \draw[hedge] (\startangle:\radius) ++ (\mxshift,\myshift) arc (\startangle:360+\endangle:\radius);

  \foreach \i in {1,...,\numbert}{
    \coordinate (u\i) at (\startangle+\i*\angle-\angle:\radius);
    \coordinate[shift={(\mxshift,\myshift)}] (v\i) at (\startangle+\i*\angle-\angle:\radius);
    \draw[oben] (u\i) -- (v\i);
  }
  \draw[oben] (u1) -- (v\numbert);
  \draw[oben] (v1) -- (u\numbert);
  \draw[oben] (\startangle:\radius) arc (\startangle:360+\endangle:\radius);

  \foreach \i in {1,...,\numbert}{
    \node[hvertex] at (u\i){};
    \node[hvertex] at (v\i){};
  }
\end{scope}

\begin{scope}[shift={(5.5,-0.5)}]
\def\radius{1.2cm}
\def\labeldist{10pt}
\antiweb{8}{2}
\end{scope}
\end{tikzpicture}
\caption{Two views of the M\"obius ladder on $8$ vertices}\label{moebius8fig}
\end{figure}
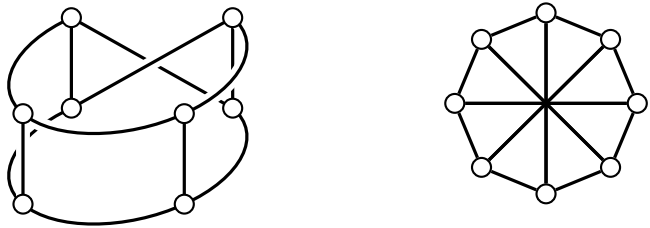

Shepherd proved:
\begin{theorem}[Shepherd~\cite{Shepherd95}]\label{shepthm}
Let $G$ be a near-bipartite graph. Then $G$ is $t$-perfect if and only if
\begin{enumerate}[\rm (i)]
\item $G$ contains no induced odd wheel; and
\item $G$ contains no induced prime antiweb other than possibly an odd hole.
\end{enumerate}
\end{theorem}

Holm, Torres and Wagler~\cite{HTW10} gave a neat characterisation 
of $t$-perfect antiwebs. For us, however, a direct implication of the proof
of that characterisation
is more interesting: an antiweb
is $t$-perfect if and only if it does not contain any even M\"obius ladder,
or any of $\aweb{7}{1}$, $\aweb{10}{2}$, $\aweb{13}{3}$, $\aweb{13}{4}$, $\aweb{17}{4}$
and $\aweb{19}{7}$ as an induced subgraph.
We may omit $\aweb{17}{4}$ from that list as it contains an induced $\aweb{13}{3}$.
Combining the theorem of Holm et al.\ with Theorem~\ref{shepthm}
one obtains:

\begin{proposition}\label{propnearbip}
A near-bipartite graph is $t$-perfect if and only if 
it does not contain any odd wheel, any even M\"obius ladder,
or any of $\aweb{7}{1}$, $\aweb{10}{2}$, $\aweb{13}{3}$, $\aweb{13}{4}$
and $\aweb{19}{7}$ as an induced subgraph.
\end{proposition}

\section{Minimally $t$-imperfect antiwebs}\label{secMinAweb}

For any characterisation of $t$-perfection in 
\emph{minimally $t$-imperfect}, that is, all graphs that are $t$-imperfect but 
whose proper $t$-minors are $t$-perfect.
Even M\"obius ladders and odd wheels, for instance, are known to be minimally $t$-imperfect.
This follows from
the result of Fonlupt and Uhry~\cite{FonUhr82} 
that \emph{almost bipartite} graphs are $t$-perfect; a graph is almost 
bipartite if it contains a vertex whose deletion renders it bipartite.
It is easy to check that any proper $t$-minor of an even M\"obius ladder or an odd wheel
is almost bipartite. 

All the other forbidden $t$-minors in Theorem~\ref{thmmainres} or Proposition~\ref{propnearbip}
are minimally $t$-imperfect, too.
That $C^2_7$ is minimally $t$-imperfect is proved in \cite{tperfect}. 
There, also minimality for $C^2_{10}$ is shown, which allows us to 
verify that~\aweb{10}{2} is minimally $t$-imperfect as well. 
Indeed, for this we first observe that~\aweb{10}{2} can be obtained
from $C^2_{10}$ by adding diagonals of the underlying $10$-cycle. 
The second necessary observation is that any two vertices directly 
opposite in the $10$-cycle form a so called \emph{odd pair}: 
any induced path between them has odd length. 
Minimality now follows from the result of  Fonlupt and Hadjar~\cite{FonHad02}
that adding an edge between the vertices of an odd pair preserves $t$-perfection.

\medskip
In this section, we prove that  \aweb{13}{3}, \aweb{13}{4}
and \aweb{19}{7} are minimally $t$-imperfect, which was not observed before.  
As prime antiwebs these are $t$-imperfect. This 
follows from Theorem~\ref{shepthm} but can also be seen directly
by observing that the vector $x\equiv\tfrac{1}{3}$
lies in $\tstab$ but not in $\ssp$ for any of the three graphs.

To show that the graphs are \emph{minimally} $t$-imperfect, it suffices
to consider the $t$-minors obtained from a single vertex deletion 
or from a single $t$-contraction. If these are $t$-perfect then the
antiweb is minimally $t$-imperfect. 

Trotter gave necessary and sufficient conditions when an antiweb contains
another antiweb:
\begin{theorem}[Trotter~\cite{Trotter75}]\label{trotthm}
$\aweb{n'}{k'}$ is an induced subgraph of $\aweb{n}{k}$ if and only if 
\begin{equation*} 
n(k'+1) \geq n' (k+1)  \text{ and }
nk'\leq n'k.
\end{equation*}
\end{theorem}

We fix the vertex set of any antiweb $\aweb{n}{k}$
to be $\{0,1, \dots , n-1 \}$, so that $ij$ is an edge of $\aweb{n}{k}$
if and only if $|i-j|\mod n>k$.

\begin{proposition}
The antiweb $\aweb{13}{3}$ is minimally $t$-imperfect.
\end{proposition}
\begin{proof}
For $\aweb{13}{3}$ to be minimally $t$-imperfect,
every proper  $t$-minor $\aweb{13}{3}$ needs to be $t$-perfect.
As no vertex of $\aweb{13}{3}$ has a stable neighbourhood, 
any proper $t$-minor is a $t$-minor of a proper induced subgraph $H$ of $\aweb{13}{3}$. 
Thus, it suffices to show that any such $H$ is $t$-perfect. 

By Proposition~\ref{propnearbip},  $H$ is $t$-perfect unless
it contains an odd wheel or
 one of $\aweb{7}{1}$, \aweb{8}{2} or \aweb{10}{2} as an induced subgraph.
Since the neighbourhood of every vertex is stable, $H$ cannot contain any wheel. 
For the other graphs, we check the inequalities of Theorem~\ref{trotthm} and see that 
none can be contained in $H$. Thus, $H$ is $t$-perfect and $\aweb{13}{3}$
therefore minimally $t$-imperfect.
\end{proof}

\begin{figure}[ht]
\centering
\begin{tikzpicture}

%Radius zum Vergroessern
\def\radius{2cm}
\def\nodelabels{$\converti{\i}$}
\def\labeldist{12pt}
\antiweb{13}{4}

\begin{scope}[shift={(6,0)}]
\def\angle{360/13}

\foreach \i in {2,3,5,6,7,8,10,11}{  
  \pgfmathtruncatemacro{\myresult}{mod(\i*5,13)}
  \node[hvertex] (v\myresult) at (90+\i*\angle:\radius){};
  \node at (90+\i*\angle:\labeldist+\radius){$\converti{\i}$};
}

\foreach \i in {0,5,6,7,8}{  
  \node[hvertex] (v\i) at (90:\radius){};
}
\node at (90:\labeldist+\radius){$\tilde 0$};

\foreach \i in {0,...,12}{  
  \foreach \j in {5,6}{
    \pgfmathtruncatemacro{\othervx}{mod(\i+\j,13)}
    \draw[hedge] (v\i) -- (v\othervx);
  }
}
\end{scope}

\def\nodelabels{}
\end{tikzpicture}

\caption{Antiweb $\aweb{13}{4}$, and its $t$-minor obtained by a $t$-contraction at~$0$}
\label{antiwebfig(4,13)}
\end{figure}
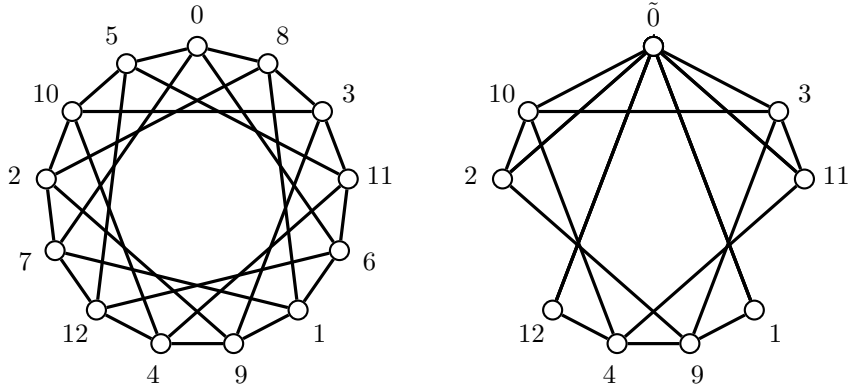
\begin{proposition}
The antiweb $\aweb{13}{4}$ is minimally $t$-imperfect.
\end{proposition}
\begin{proof}
By Proposition~\ref{propnearbip}, any proper induced subgraph 
of \aweb{13}{4} that is $t$-imperfect contains one of 
$\aweb{7}{1}$, $\aweb{8}{2}$, or $\aweb{10}{2}$ as an induced subgraph;
note that \aweb{13}{4} does not contain odd wheels.
However, routine calculation and Theorem~\ref{trotthm}
show that \aweb{13}{4} contains neither of these. Therefore, 
deleting any vertex in \aweb{13}{4} always results in a $t$-perfect graph.

It remains to consider the graphs obtained from \aweb{13}{4} by a single 
$t$-contraction. By symmetry, it suffices check whether the graph 
$H$ obtained by $t$-contraction at~$0$ is $t$-perfect;
see Figure~\ref{antiwebfig(4,13)}. 
Denote by $\tilde{0}$ the new vertex that resulted from the contraction.

The graph $H$ is still near-bipartite and still devoid of odd wheels. 
Thus, by Proposition~\ref{propnearbip}, it
is $t$-perfect unless it contains  $\aweb{7}{1}$ and $\aweb{8}{2}$
as an induced subgraph---all the $t$-imperfect antiwebs  of Proposition~\ref{propnearbip}
are too large for the nine-vertex graph $H$.

Now, $\aweb{7}{1}$ is $4$-regular but $H$ only contains five vertices of 
degree at least~$4$. Similarly, $\aweb{8}{2}$ is $3$-regular but 
two of the nine vertices of $H$, namely~$1$ and~$12$, have 
degree~$2$. We see that neither of the two antiwebs can be contained in $H$, 
so that $H$ is $t$-perfect and, thus, $\aweb{13}{4}$ minimally $t$-imperfect. 
\end{proof}

\begin{figure}[ht]
\centering
\begin{tikzpicture}[scale=0.8,auto]

%Radius zum Vergroessern
\def\radius{2.4cm}
\def\nodelabels{$\converti{\i}$}
\def\labeldist{14pt}
\antiweb{19}{7}

\begin{scope}[shift={(7,0)}]
\def\angle{360/19}

\foreach \i in {2,3,4,5,7,8,9,10,11,12,14,15,16,17}{  
  \pgfmathtruncatemacro{\myresult}{mod(\i*8,19)}
  \node[hvertex] (v\myresult) at (90+\i*\angle:\radius){};
  \node at (90+\i*\angle:\labeldist+\radius){$\converti{\i}$};
}

\foreach \i in {0,8,9,10,11}{  
  \node[hvertex] (v\i) at (90:\radius){};
}
\node at (90:\labeldist+\radius){$\tilde 0$};

\foreach \i in {0,...,18}{  
  \foreach \j in {8,9}{
    \pgfmathtruncatemacro{\othervx}{mod(\i+\j,19)}
    \draw[hedge] (v\i) -- (v\othervx);
  }
}
\end{scope}

\end{tikzpicture}

\caption{Antiweb $\aweb{19}{7}$, and its $t$-minor obtained by $t$-contraction at~$0$}
\label{antiwebfig(7,19)}
\end{figure}

\begin{proposition}
The antiweb $\aweb{19}{7}$ is minimally $t$-imperfect.
\end{proposition}
\begin{proof}
We claim that any proper induced subgraph of \aweb{19}{7} is $t$-perfect. 
Indeed, as \aweb{19}{7} does not contain any induced odd wheel, 
this follows from Proposition~\ref{propnearbip}, unless \aweb{19}{7} 
contains one of $\aweb{7}{1}, \aweb{8}{2},  \aweb{10}{2}, \aweb{12}{4}, \aweb{13}{3}, \aweb{13}{4}$, or $\aweb{16}{6}$ as an induced subgraph.
We can  easily verify with Theorem~\ref{trotthm} that this is not the case.

It remains to check that any $t$-contraction in \aweb{19}{7} yields a $t$-perfect
graph, too. By symmetry, we may restrict ourselves to a $t$-contraction
at the vertex~$0$. Let $H$ be the resulting graph, and let~$\tilde 0$ be 
the new vertex; see Figure~\ref{antiwebfig(7,19)}.

The graph $H$ is a near-bipartite graph on~$15$ vertices. It does not contain
any odd wheel as an induced subgraph. Thus, by  Proposition~\ref{propnearbip}, 
$H$ is $t$-perfect unless it has an induced subgraph $A$ that is isomorphic 
to a graph in 
\[
\mathcal A:= \{\aweb{7}{1}, \aweb{8}{2}, \aweb{10}{2}, \aweb{12}{4}, \aweb{13}{3}, \aweb{13}{4}\}.
\]
Since this is not the case for $\aweb{19}{7}$, we may assume that  $\tilde{0} \in V(A)$.

Note that the graphs \aweb{7}{1}, \aweb{10}{2}, \aweb{13}{3} and \aweb{13}{4}
have minimum degree at least~$4$. Yet, $\tilde 0$ has only two neighbours of
degree~$4$ or more (namely, $3$ and $16$). Thus, neither of these 
four antiwebs can occur as an induced subgraph in $H$. 

It remains to consider the case when $H$ contains an induced subgraph $A$ that is isomorphic
to $\aweb{8}{2}$ or to $\aweb{12}{4}$, both of which are $3$-regular graphs. 
In particular, $A$ is then contained in $H'=H-\{1,18\}$ as the vertices~$1$ and~$18$
have degree~$2$. 

As $H'$ has only~$13$ vertices, $A$ cannot be isomorphic to~$\aweb{12}{4}$
since deleting any single vertex of $H'$ never yields a $3$-regular graph. 
That leaves only $A=\aweb{8}{2}$. 

Since \aweb{8}{2} is $3$-regular, we need to delete exactly one of the 
four neighbours of~$\tilde 0$ in $H'$. Suppose this is the vertex~$3$. 
Then,  $12$ has degree~$2$ and thus cannot be part of~$A$. 
Deleting~$12$ as well leads to vertex~$2$ having degree~$2$, which 
thereby is also excluded from~$A$. This, however, is impossible
as~$2$ is one of the three remaining neighbours of~$\tilde 0$.

By symmetry, we may therefore assume that the neighbours of~$\tilde 0$
in $A$ are precisely $2,3,16$. That~$17$ is not part of~$A$ entails 
that the vertex~$7$ has degree~$2$ and thus cannot lie in $A$ either. 
Then, however, $16\in V(A)$ has degree~$2$ as well, which is impossible.
\end{proof}

\section{Harmonious cutsets}\label{sectionharm}

We investigate the structure of minimally $t$-imperfect
graphs, whether they are $P_5$-free or not. We hope 
this  more general setting might prove useful in subsequent
research.

A structural feature that may never appear in a minimally $t$-imperfect
graph $G$ is a \emph{clique separator}: any clique $K$ of $G$ so that $G-K$
is not connected. 

\begin{lemma}[Chv\'atal~\cite{Chvatal75}; Gerards~\cite{Gerards89}]\label{cliqueseplem}
No minimally $t$-imperfect graph  contains a clique separator.
\end{lemma}

A generalisation of clique separators was introduced 
by Chudnovsky et al.~\cite{CRST10} in the context of colouring $K_4$-free
graphs without odd holes. 
A tuple $(X_1,\ldots,X_s)$ of disjoint subsets of the vertex set of a graph $G$ is 
\emph{$G$-harmonious} if 
\begin{itemize}
\item any induced path with one endvertex in $X_i$ and the other in $X_j$
has even length if and only if $i=j$; and
\item if $s\geq 3$ then $X_1,\ldots, X_s$ are pairwise complete to each other.
\end{itemize}

A pair of subgraphs $\{ G_1, G_2 \}$ of $G=(V,E)$ is a \emph{separation} of $G$ if $V(G_1)\cup V(G_2)=V$ and $G$ has no edge between $V(G_1) \setminus V(G_2)$ and $V(G_2) \setminus V(G_1)$.
If both $V(G_1) \setminus V(G_2)$ and $V(G_2) \setminus V(G_1)$ are non-empty, the separation is \emph{proper}. 

A vertex set $X$ is called a \emph{harmonious cutset} if there
is a proper separation $(G_1,G_2)$ of $G$ so that $X=V(G_1)\cap V(G_2)$
and if there exists a partition $X=(X_1,\ldots,X_s)$ so that $(X_1,\ldots,X_s)$
is $G$-harmonious.	

We prove:
\begin{lemma}\label{noharmlem}
If a $t$-imperfect graph contains a harmonious
cutset then it also contains a proper induced  subgraph that is $t$-imperfect.
In particular,
no minimally $t$-imperfect graph  admits a harmonious cutset.
\end{lemma}

For the proof we need a bit of preparation.

\begin{lemma}\label{nestedlem}
Let $S_1\subsetneq\ldots\subsetneq S_k$ and 
$T_1\subsetneq\ldots\subsetneq T_\ell$ be nested subsets of 
a finite set $V$.
Let $\sigma:=\sum_{i=1}^k\lambda_i\charf{S_i}$ and $\tau:=\sum_{j=1}^\ell\mu_j\charf{T_j}$
be two convex combinations in $\R^V$ with non-zero coefficients.
If $\sigma=\tau$
then $k=\ell$, $\lambda_i=\mu_i$ and $S_i=T_i$ for all $i=1,\ldots,k$.
\end{lemma}

The lemma is not new. It appears in the context of submodular functions, where
it may be seen to assert that the \emph{Lov\'asz extension} of a set-function 
is well-defined; see Lov\'asz~\cite{Lov83}. For the sake of completeness, we give 
a proof here.

\begin{proof}
By allowing $\lambda_1$ and $\mu_1$ to be~$0$, we may clearly assume that $S_1=\emptyset=T_1$. 
Moreover, if two elements $u,v\in V$ always appear together in the sets $S_i$, $T_j$ 
then we may omit one of $u,v$ from all the sets. So, in particular, 
we may assume $S_2$ and $T_2$ 
to be singleton-sets.
 
Let $s$ be the unique element of $S_2$. 
Then $\sum_{i=2}^k\lambda_i=\sigma_s=\tau_s\leq\sum_{j=2}^\ell\mu_j$.
By symmetry, we also get $\sum_{i=2}^k\lambda_i\geq \sum_{j=2}^\ell\mu_j$, and thus we have equality. 
We deduce that $T_2=\{s\}$, and that $\lambda_1=\mu_1$ as 
$\lambda_1=1-\sum_{i=2}^k\lambda_i=1-\sum_{j=2}^\ell\mu_j=\mu_1$.
Then 
\[
(\lambda_1+\lambda_2)\charf{S_1}+\sum_{i=3}^k\lambda_i\charf{S_i\sm\{s\}}
= (\mu_1+\mu_2)\charf{T_1}+\sum_{j=3}^\ell\mu_j\charf{T_j\sm\{s\}}
\]
are two convex combinations. Induction on $|S_k|$ now finishes the proof, where 
we also use that $\lambda_1=\mu_1$.
\end{proof}

\begin{lemma}\label{forcenested}
Let $G$ be a graph, and 
let $(X,Y)$ be a $G$-harmonious tuple (with possibly $X=\emptyset$ or $Y=\emptyset$).
If $S_1,\ldots,S_k$ are stable sets
then there are stable sets $S'_1,\ldots,S'_k$ so that
\begin{enumerate}[\rm (i)]
\item $S'_1\cap X\subseteq\ldots\subseteq S'_k\cap X$;
\item $S'_1\cap Y\supseteq\ldots\supseteq S'_k\cap Y$; and
\item $\sum_{i=1}^k\charf{S'_i}=\sum_{i=1}^k\charf{S_i}$.
\end{enumerate}
\end{lemma}
\begin{proof}
We start with two easy claims. First:
\begin{equation}\label{Xsmall}
\begin{minipage}[c]{0.8\textwidth}\em
For any two stable sets $S,T$ there are stable sets $S'$
and $T'$ such that $\charf S+\charf T=\charf{S'}+\charf{T'}$ 
and $S'\cap X\subseteq T'\cap X$.
\end{minipage}\ignorespacesafterend 
\end{equation} 
Indeed, assume there is an $x\in (S\cap X)\sm T$. 
Denote by $K$ the component
of the induced 
graph $G[S\cup T]$ that contains $x$, and 
consider the symmetric differences $\tilde S=S\triangle K$ and 
$\tilde T=T\triangle K$. 
Clearly, $\charf{S}+\charf{T}=\charf{\tilde S}+\charf{\tilde T}$. 
Moreover, $K$ meets $X$ only in $S$ as otherwise $K$ would contain an induced 
$x$--$(T\cap X)$ path, which then has necessarily odd length. This, however, is impossible
as $(X,Y)$ is $G$-harmonious. Therefore, $x\notin \tilde S\cap X\subset S\cap X$. 
By repeating this exchange argument for any remaining $x'\in(\tilde S\cap X)\sm \tilde T$,
we arrive at the desired stable sets $S'$ and $T'$. 
This proves~\eqref{Xsmall}.

We need a second,  similar assertion:
\begin{equation}\label{Ylarge}
\begin{minipage}[c]{0.8\textwidth}\em
For any two stable sets $S,T$ with $S\cap X\subseteq T\cap X$
there are stable sets $S'$ and $T'$
such that $\charf S+\charf T=\charf{S'}+\charf{T'}$, 
$S'\cap X= S\cap X$ and $S'\cap Y\supseteq T'\cap Y$.
\end{minipage}\ignorespacesafterend 
\end{equation} 
To see this, assume there is a $y\in (T\cap Y)\sm S$, and
let $K$ be the component of $G[S\cup T]$ containing $y$,
and set $\tilde S=S\triangle K$ and $\tilde T=T\triangle K$. 
The component $K$ may not meet $T\cap X$, as
then it would contain an induced $y$--$(T\cap X)$ path. This path would
have even length, contradicting the definition of a $G$-harmonious tuple.
As above, we see, moreover, that $K$ meets $Y$
only in $T$; otherwise there would be an induced odd $y$--$(S\cap Y)$ path,
which is impossible. Thus, $\tilde S$, $\tilde T$ satisfy the first two 
conditions we want to have for $S',T'$, while $(\tilde T\cap Y)\sm\tilde S$
is smaller than $(T\cap Y)\sm S$. Again repeating the argument yields $S',T'$
as desired.
This proves~\eqref{Ylarge}.

\medskip
We now apply~\eqref{Xsmall} iteratively 
to $S_1$ (as $S$) and each of $S_2,\ldots,S_k$ (as $T$) 
in order to obtain stable sets $R_1,\ldots, R_k$ with $R_1\cap X\subseteq R_i\cap X$ for 
every $i=2,\ldots, k$ and $\sum_{i=1}^k\charf{S_i}=\sum_{i=1}^k\charf{R_i}$.
We continue applying~\eqref{Xsmall}, first to $R_2$ and each of $R_3,\ldots, R_k$,
then to the resulting $R'_3$ and each of $R'_4,\ldots, R'_k$ and so on, 
until we arrive at  stable sets $T_1,\ldots, T_k$ 
 with 
$\sum_{i=1}^k\charf{S_i}=\sum_{i=1}^k\charf{T_i}$
that are nested on $X$: $T_1\cap X\subseteq\ldots\subseteq T_k\cap X$.

In a similar way, we use~\eqref{Ylarge} to force the stable sets to become
nested on $Y$ as well. First, we apply~\eqref{Ylarge} to $T_1$ (as $S$)
and to each of $T_2,\ldots, T_k$ (as $T$), then to the resulting $T'_3$ 
and each of $T'_4,\ldots,T'_k$, and so on. Proceeding in this manner, we 
obtain the desired stable sets $S'_1,\ldots, S'_k$. 
\end{proof}

\begin{lemma}\label{harmSSP}
Let $(G_1,G_2)$ be a proper separation of a graph
$G$ so that $X=V(G_1)\cap V(G_2)$
is a harmonious cutset. 
Let $z\in\mathbb Q^{V(G)}$ be so that $z\rstr{G_1}\in\ssp(G_1)$ and 
$z\rstr{G_2}\in\ssp(G_2)$. Then $z\in \ssp(G)$.
\end{lemma}
The  lemma generalises the result by Chudnovsky et al.~\cite{CRST10} 
that $G=G_1\cup G_2$ is  $4$-colourable if $G_1$ and $G_2$ are $4$-colourable. 
\begin{proof}[Proof of Lemma~\ref{harmSSP}]
Let $(X_1, \ldots,X_s)$ be a $G$-harmonious partition of $X$.
As $z\rstr{G_j}\in \ssp(G_j)$, for $j=1,2$,  we can express $z\rstr{G_1}$ as a convex combination 
of stable sets $S_{1},\ldots, S_{m}$ of $G_1$, and $z\rstr{G_2}$
as a convex combination of stable sets 
$T_{1},\ldots, T_{m'}$ of $G_2$. Since $z$ is a rational vector, 
we may even assume that
\[
z\rstr{G_1}=\frac{1}{m}\sum_{i=1}^m\charf{S_{i}}\text{ and }z\rstr{G_2}=\frac{1}{m}\sum_{i=1}^m\charf{T_{i}}.
\]
Indeed, this can be achieved by repeating stable sets.

\medskip
We first treat the case when $s\leq 2$.
If $s=1$, then set $X_2=\emptyset$,
so that whenever $s\leq 2$, we have $X=X_1\cup X_2$.
  
Using Lemma~\ref{forcenested}, we find stable sets 
$S'_{1},\ldots, S'_{m}$ of $G_1$ so that 
$z\rstr{G_1}=\frac{1}{m}\sum_{i=1}^m\charf{S'_{i}}$ and 
\[
S'_{1}\cap X_1\subseteq\ldots\subseteq S'_{m}\cap X_1,
\text{ and }S'_{1}\cap X_2\supseteq\ldots\supseteq S'_{m}\cap X_2
\] 
holds. Analogously, we obtain a convex combination $z\rstr{G_2}=\frac{1}{m}\sum_{i=1}^m\charf{T'_{i}}$
of
stable sets $T'_1,\ldots,T'_m$ of $G_2$ that are increasingly nested on $X_1$ 
and decreasingly nested on $X_2$.

Define $\overline S_{1}\subsetneq \ldots\subsetneq \overline S_{k}$
to be the distinct restrictions of the sets $S'_{i}$ to $X_1$. More formally, 
let $1= i_1<\ldots < i_k<i_{k+1}=m+1$ be so that 
\[
\overline S_{t} = S'_{i}\cap X_1\text{ for all }i_{t}\leq i< i_{t+1}
\]
We set, moreover, $\lambda_{t}=\tfrac{1}{m}(i_{t+1}-i_t)$. Equivalently, 
 $m\lambda_{t}$ is  the number of $S'_{i}$ with $\overline S_{t} = S'_{i}\cap X_1$.
Then $z\rstr{X_1}=\sum_{t=1}^k\lambda_{t}\charf{\overline S_{t}}$ is a convex combination.

We do exactly the same  in $G_2$ in order to 
obtain $z\rstr{X_1}=\sum_{t=1}^k\mu_{t}\charf{\overline T_{t}}$, where 
the sets $\overline T_{t}$ are the distinct restrictions of the $T'_i$ to $X_1$.
With Lemma~\ref{nestedlem}, we deduce first that
$\overline S_t=\overline T_t$ and $\lambda_t=\mu_t$ for all $t$, from which we get that 
\[
S'_i\cap X_1 = T'_i\cap X_1\text{ for all }i=1,\ldots, m.
\]
The same argument, only applied to the restrictions of $S'_i$ and of $T'_i$ to $X_2$,
yields that also 
\[
S'_i\cap X_2 = T'_i\cap X_2\text{ for all }i=1,\ldots, m.
\]
Thus, $R_i=S'_i\cup T'_i$ is, for every $i=1,\ldots, m$, a stable set of $G$. Consequently,
$z=\tfrac{1}{m}\sum_{i=1}^m\charf{R_i}$ is a convex combination of stable sets
and thus a point of $\ssp(G)$.

\medskip

It remains to treat the case when the harmonious cutset has at least three parts,
that is, when $s\geq 3$. 
We claim that 
there are sets $\mathcal S_0,\mathcal S_1,\ldots,\mathcal S_s$ 
of stable sets of $G_1$
so that 
\begin{enumerate}[\rm (a)]
\item\label{conda} $z\rstr{G_1}=\frac{1}{m}\sum_{j=0}^s\sum_{S\in\mathcal S_j}\charf{S}$
and $\sum_{j=0}^s|\mathcal S_j|=m$;
\item\label{condb} for $j=1,\ldots, s$ if $S\in\mathcal S_j$ then $X_j\cap S$ is non-empty; and
\item\label{condc} for $j=0,\ldots, s$ if $S,T\in\mathcal S_j$ then $X_j\cap S\subseteq X_j\cap T$
or $X_j\cap S\supseteq X_j\cap T$.
\end{enumerate}
Moreover, there are analogous sets $\mathcal T_0,\mathcal T_1,\ldots,\mathcal T_s$
for $G_2$.

\medskip
To prove the claim note first that each $S_i$ meets at most one of the sets $X_j$ 
as each two induce a complete bipartite graph. Therefore, we can 
partition $\{S_1,\ldots,S_m\}$ into sets $\mathcal S'_0,\ldots,\mathcal S'_s$
so that~\eqref{conda} and~\eqref{condb} are satisfied.
Next, we apply Lemma~\ref{forcenested} to each $\mathcal S_j'$ 
and $(X_j,\emptyset)$ in order
to obtain sets $\mathcal S''_j$ that satisfy~\eqref{conda} and~\eqref{condc}
but not necessarily~\eqref{condb};
property~\eqref{conda} still holds as Lemma~\ref{forcenested} guarantees
$ \sum_{S\in \mathcal S'_j} \charf{S}=  \sum_{S\in \mathcal S''_j} \charf{S}$
for each $j$. 
If~\eqref{condb} is violated, then only because for some $j\neq 0$ 
there is $S\in\mathcal S''_j$ that is not only disjoint from $X_j$ 
but also from all other $X_{j'}$. 
In order to repair~\eqref{condb} we 
remove all stable sets $S$ in $\bigcup_{j=1}^s\mathcal S''_j$
that are disjoint from $\bigcup_{j=1}^s X_j$ from their respective 
sets and add them to  $\mathcal S''_0$. The resulting sets 
$\mathcal S_0,\mathcal S_1,\ldots,\mathcal S_s$ 
then satisfy~\eqref{conda}--\eqref{condc}. 
The proof for the $\mathcal T_j$ is the same.
\medskip

As a consequence of~\eqref{conda} and~\eqref{condb}
it follows  for $j=0,1,\ldots, s$ that  
\begin{equation}\label{zonXj}
\sum_{S\in \mathcal{{S}}_j} \charf{S \cap X_j} =
m\cdot z\rstr{X_j}
= \sum_{T\in \mathcal{{T}}_j} \charf{T \cap X_j}
\end{equation}

Now, consider $j\neq 0$. Then, by~\eqref{condb} and~\eqref{condc},
there is a vertex $v\in X_j$ that lies in every $S\in\mathcal S_j$. 
Thus, we have 
$\sum_{S\in \mathcal S_j} \charf{S}(v)= |\mathcal S_j|$.

Evaluating~\eqref{zonXj} at $v\in X_j$, we obtain 
\[
|\mathcal S_j|= m\cdot z(v)=\sum_{T\in \mathcal{{T}}_j} \charf{T}(v)\leq |\mathcal T_j|.
\]
Reversing the roles of $\mathcal S_j$ and $\mathcal T_j$, we
also get $|\mathcal T_j|\leq|\mathcal S_j|$, and thus that 
$|\mathcal T_j|=|\mathcal S_j|$, as long as $j\neq 0$. 
That this also holds for $j=0$ follows from 
$m=\sum_{j=0}^s |\mathcal S_j| = \sum_{j=0}^s |\mathcal T_j|$,
so that we get  
$
m_j:=\vert\mathcal S_j\vert = \vert\mathcal T_j\vert \mbox{ for every } j=0,1, \ldots, s.
$

Together with~\eqref{zonXj} this implies, in particular, that 
\[
\frac{1}{m_j}\sum_{S\in\mathcal S_j}\charf{S\cap X_j}=
\frac{1}{m_j}\sum_{T\in\mathcal T_j}\charf{T\cap X_j}
\]
We may, therefore,  define a vector $y^j$ on $V(G)$ by
setting 
\begin{equation}\label{yjdef}
y^j \rstr{G_1} := \frac{1}{m_j}\sum_{S\in\mathcal S_j}\charf{S}
\emtext{ and }
y^j \rstr{G_2} := \frac{1}{m_j}\sum_{T\in\mathcal T_j}\charf{T}
\end{equation}

For any $j=0,\ldots, s$, 
define $G^j=G-\bigcup_{r\neq j}X_r$, and observe that $X_j$ is a harmonious cutset of $G^j$
consisting of only one part. (That is, $X_j$ is $G^j$-harmonious.)  
Moreover, as~\eqref{yjdef} shows,
the restriction of $y^j$ to $G_1\cap G^j$
lies in $\ssp(G_1\cap G^j)$, while the restriction to $G_2\cap G^j$ lies in $\ssp(G_2\cap G^j)$.
Thus, we can apply the first part of this proof, when $s\leq 2$, in order to
deduce that $y^j\in \ssp(G^j)\subseteq\ssp(G)$. 

To finish the proof we observe, with~\eqref{conda} and~\eqref{yjdef}, that 
\[
z=\sum_{j=0}^s\frac{m_j}{m}y^j
\]
As, by~\eqref{conda}, $\sum_{j=0}^sm_j=m$, this 
means that $z$ is a convex combination of points in $\ssp(G)$, 
and thus itself an element of $\ssp(G)$.
\end{proof}

\begin{corollary}
Let $(G_1,G_2)$ be a proper separation of $G$ so that $X=V(G_1)\cap V(G_2)$
is a harmonious cutset. Then $G$ is $t$-perfect if and only if $G_1$
and $G_2$ are $t$-perfect.
\end{corollary}
\begin{proof}
Assume that $G_1$ and $G_2$ are $t$-perfect, and consider a rational
point $z\in\tstab(G)$.
Then $z|G_1\in\ssp(G_1)$ and $z|G_2\in\ssp(G_2)$, which means
that Lemma~\ref{harmSSP} yields $z\in\ssp(G)$. Since this is true for all 
rational $z$ it extends to real $z$ as well. 
\end{proof}

The corollary directly implies Lemma~\ref{noharmlem}.

\section{$P_5$-free graphs}\label{secmainres}

Let $\mathcal F$ be the set of graphs consisting of $P_5$,
$K_4$, $W_5$, $C^2_7$, $\aweb{10}{2}$ and $\aweb{13}{3}$ 
together with the three graphs in Figure~\ref{K4fig}.
Note that the latter three graphs all contain $K_4$ as a $t$-minor:
for~(a) and~(b) $K_4$ is obtained by a $t$-contraction at any vertex of degree~$2$, 
while for~(c) both vertices of degree~$2$ need to be $t$-contracted.
In particular, every graph in $\mathcal F$ besides $P_5$ is $t$-imperfect.
We say that a graph is \emph{$\mathcal F$-free} if it contains 
none of the graphs in $\mathcal F$ as an induced subgraph.

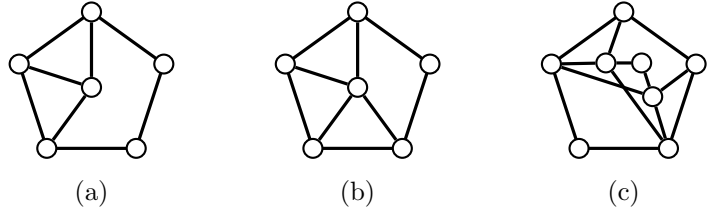
\begin{figure}[ht]
\centering
\begin{tikzpicture}[scale=1,auto]

\def\radius{1cm}
\def\angle{360/5}
\def\captionpos{(270:1.4*\radius)}

\def\hshift{3.5cm}

\begin{scope}[shift={(0,0)}]
\drawfivehole
\draw[hedge] (c) -- (v0);
\draw[hedge] (c) -- (v1);
\draw[hedge] (c) -- (v2);
\node at \captionpos {(a)};
\end{scope}

\begin{scope}[shift={(\hshift,0)}]
\drawfivehole
\draw[hedge] (c) -- (v0);
\draw[hedge] (c) -- (v1);
\draw[hedge] (c) -- (v2);
\draw[hedge] (c) -- (v3);
\node at \captionpos {(b)};
\end{scope}

\def\uvshift{0.4}

\begin{scope}[shift={(2*\hshift,0)}]
\drawfive
\node[hvertex] (v) at (-18:\uvshift*\radius){};
\node[hvertex] (u) at (126:\uvshift*\radius){};
\draw[hedge] (v0) -- (u) -- (v1);
\draw[hedge] (u) -- (v3);
\draw[hedge] (v3) -- (v) -- (v4);
\draw[hedge] (v) -- (v1);
\node[hvertex] (x) at (54:\uvshift*\radius){};
\draw[hedge] (u) -- (x) -- (v);
\node at \captionpos {(c)};
\end{scope}

\end{tikzpicture}

\caption{Three graphs that $t$-contract to $K_4$}
\label{K4fig}
\end{figure}

We prove a lemma that implies directly Theorem~\ref{thmmainres}:
\begin{lemma}\label{mainlem}
Any $\mathcal F$-free graph is $t$-perfect.
\end{lemma}

We first examine how a vertex may 
position itself relative to a $5$-cycle in an $\mathcal F$-free graph.

\begin{lemma}\label{nbtypeslem}
Let $G$ be an $\mathcal F$-free graph.
If $v$ is a neighbour of a $5$-hole $C$ in $G$ then 
$v$ has either exactly two neighbours in $C$, and these are non-consecutive in $C$; 
or $v$ has exactly three neighbours in $C$, and these are not all consecutive.
\end{lemma}

\begin{figure}[ht]
\centering
\begin{tikzpicture}[scale=0.8,auto]

\def\radius{1cm}
\def\angle{360/5}
\def\captionpos{(270:1.4*\radius)}

\def\vshift{3.5cm}
\def\hshift{-3.3cm}

\begin{scope}
\drawfivehole
\node at \captionpos {(a) \cmark};
\end{scope}

\begin{scope}[shift={(\vshift,0)}]
\drawfivehole
\draw[hedge] (c) -- (v0);
\node at \captionpos {(b) \xmark};
\end{scope}

\begin{scope}[shift={(2*\vshift,0)}]
\drawfivehole
\draw[hedge] (c) -- (v0);
\draw[hedge] (c) -- (v1);
\node at \captionpos {(c) \xmark};
\end{scope}

\begin{scope}[shift={(3*\vshift,0)}]
\drawfivehole
\draw[hedge] (c) -- (v0);
\draw[hedge] (c) -- (v2);
\node at \captionpos {(d) \cmark};
\end{scope}

\begin{scope}[shift={(0,\hshift)}]
\drawfivehole
\draw[hedge] (c) -- (v0);
\draw[hedge] (c) -- (v1);
\draw[hedge] (c) -- (v2);
\node at \captionpos {(e) \xmark};
\end{scope}

\begin{scope}[shift={(\vshift,\hshift)}]
\drawfivehole
\draw[hedge] (c) -- (v0);
\draw[hedge] (c) -- (v1);
\draw[hedge] (c) -- (v3);
\node at \captionpos {(f) \cmark};
\end{scope}

\begin{scope}[shift={(2*\vshift,\hshift)}]
\drawfivehole
\draw[hedge] (c) -- (v0);
\draw[hedge] (c) -- (v1);
\draw[hedge] (c) -- (v2);
\draw[hedge] (c) -- (v3);
\node at \captionpos {(g) \xmark};
\end{scope}

\begin{scope}[shift={(3*\vshift,\hshift)}]
\drawfivehole
\draw[hedge] (c) -- (v0);
\draw[hedge] (c) -- (v1);
\draw[hedge] (c) -- (v2);
\draw[hedge] (c) -- (v3);
\draw[hedge] (c) -- (v4);
\node at \captionpos {(h) \xmark};
\end{scope}

\end{tikzpicture}

\caption{The types of neighbours of a $5$-hole}
\label{nbtypesfig}
\end{figure}
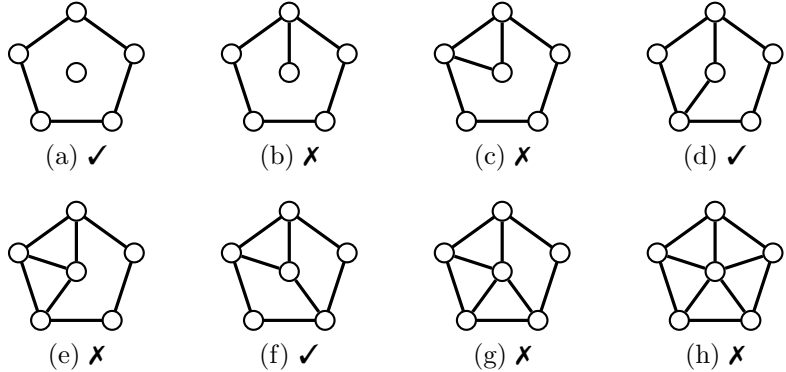

\begin{proof}
See Figure~\ref{nbtypesfig} for the possible types of neighbours (up to isomorphy). 
Of these, (b) and (c) contain an induced $P_5$; (e) and (g) are the same as (a) and (b) in Figure~\ref{K4fig}
and thus
in $\mathcal F$; (h) is $W_5$. 
Only (d) and (f) remain.
\end{proof}

\begin{lemma}\label{nbtypestwolem}
Let $G$ be an $\mathcal F$-free graph, and 
let  $u$ and $v$ be two non-adjacent vertices such that both of them have precisely three neighbours in a $5$-hole $C$. Then $u$ and $v$ have either all three or exactly two non-consecutive neighbours in $C$ in common.
 \end{lemma}

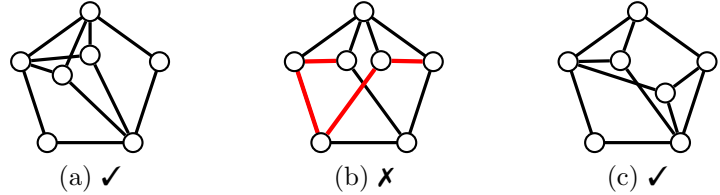
\begin{figure}[ht]
\centering
\begin{tikzpicture}[scale=0.8,auto]
\tikzstyle{redge}=[ultra thick,red]
\tikzstyle{bedge}=[ultra thick,blue]
\tikzstyle{dedge}=[ultra thick,dashed]
\tikzstyle{gedge}=[ultra thick,red]

\def\captionpos{(0,-1.3*\radius)}

\def\vshift{4.5cm}
\def\hshift{-3.5cm}
\def\radius{1.2cm}

\def\uvshift{0.4}

\begin{scope}[shift={(0,\hshift)}]
\drawfive
\node[hvertex] (u) at (90+72:\uvshift*\radius){};
\draw[hedge] (v0) -- (u) -- (v1);
\draw[hedge] (u) -- (v3);

\node[hvertex] (v) at (90:\uvshift*\radius){}; 
\draw[hedge] (v0) -- (v) -- (v1);
\draw[hedge] (v) -- (v3);
\node at \captionpos {(a) \cmark};
\end{scope}

\begin{scope}[shift={(\vshift,\hshift)}]
\drawfive
\node[hvertex] (v) at (54:\uvshift*\radius){};
\node[hvertex] (u) at (126:\uvshift*\radius){};
\draw[gedge] (u) -- (v1);
\draw[hedge] (v0) -- (u);
\draw[hedge] (u) -- (v3);
\draw[gedge] (v) -- (v4);
\draw[hedge] (v0) -- (v);
\draw[gedge] (v) -- (v2);

\draw[gedge] (v1) -- (v2);

\node at \captionpos {(b) \xmark};
\end{scope}

\begin{scope}[shift={(2*\vshift,\hshift)}]
\drawfive
\node[hvertex] (u) at (126:\uvshift*\radius){};
\draw[hedge] (v0) -- (u) -- (v1);
\draw[hedge] (u) -- (v3);

\node[hvertex] (v) at (-18:\uvshift*\radius){};
\draw[hedge] (v3) -- (v) -- (v4);
\draw[hedge] (v) -- (v1);

\node at \captionpos {(c) \cmark};
\end{scope}
\end{tikzpicture}

\caption{The possible configurations of
Lemma \ref{nbtypestwolem}}
\label{twonbfig}

\end{figure}

\begin{proof}
By Lemma~\ref{nbtypeslem}, 
both of $u$ and $v$ have to be as in~(f) of Figure~\ref{nbtypesfig}.
Figure~\ref{twonbfig} shows the possible configurations of $u$ and $v$ (up to isomorphy). 
Of these, (b) is impossible as there is an induced $P_5$---the other two 
configurations (a) and (c) may occur.
\end{proof}

A subgraph $H$ of a graph $G$ is \emph{dominating} if every vertex in $G-H$
has a neighbour in~$H$.

\begin{lemma}\label{any5hole}
Let $G$ be an $\mathcal F$-free graph. 
Then, either any $5$-hole of $G$ is dominating or $G$ contains a harmonious cutset. 
\end{lemma}

\begin{proof}
Assume  
that there is a $5$-hole $C=c_1\ldots c_5c_1$ that fails to dominate~$G$. 
Our task consists in finding a harmonious cutset. 
We first observe: 
\begin{equation}\label{xnbsclm}
\begin{minipage}[c]{0.8\textwidth}\em
Let $u\in N(C)$ be a neighbour of some $x\notin N(C)$.
Then $u$ has exactly three neighbours in $C$, not all of which are consecutive. 
\end{minipage}\ignorespacesafterend 
\end{equation}
So, such a $u$ is as in~(f) of Figure~\ref{nbtypesfig}. 
Indeed, by Lemma~\ref{nbtypeslem}, only~(d) or~(f) in Figure~\ref{nbtypesfig}
are possible. 
In the former case, we may assume that the neighbours of $u$ in $C$
are $c_1$ and $c_3$. Then, however, $xuc_1c_4c_5$ is an induced $P_5$.
This proves~\eqref{xnbsclm}.

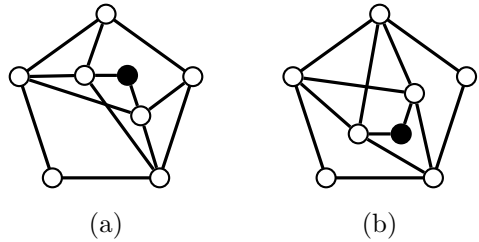
\begin{figure}[ht]
\centering
\begin{tikzpicture}[scale=0.8,auto]
\tikzstyle{redge}=[ultra thick,red]
\tikzstyle{bedge}=[ultra thick,blue]
\tikzstyle{dedge}=[ultra thick,dashed]

\def\captionpos{(0,-2)}

\def\vshift{4.5cm}
\def\hshift{-3.3cm}
\def\radius{1.5cm}

\def\uvshift{0.4}

%\begin{scope}
%\drawfive
%\node[hvertex] (v) at (54:\uvshift*\radius){};
%\node[hvertex] (u) at (126:\uvshift*\radius){};
%\draw[hedge] (v0) -- (u) -- (v1);
%\draw[hedge] (u) -- (v3);
%\draw[hedge] (v0) -- (v) -- (v4);
%\draw[hedge] (v) -- (v2);
%
%\node[hvertex,fill=black] (x) at (270:\uvshift*\radius){};
%\draw[hedge] (u) -- (x) -- (v);
%\node at \captionpos {(a)};
%\end{scope}

\begin{scope}[shift={(\vshift,0)}]
\drawfive
\node[hvertex] (v) at (-18:\uvshift*\radius){};
\node[hvertex] (u) at (126:\uvshift*\radius){};
\draw[hedge] (v0) -- (u) -- (v1);
\draw[hedge] (u) -- (v3);
\draw[hedge] (v3) -- (v) -- (v4);
\draw[hedge] (v) -- (v1);
\node[hvertex,fill=black] (x) at (54:\uvshift*\radius){};
\draw[hedge] (u) -- (x) -- (v);
\node at \captionpos {(a)};
\end{scope}

\begin{scope}[shift={(2*\vshift,0)}]
\drawfive
\def\fifth{360/5}
\node[hvertex] (v) at (90-\fifth:\uvshift*\radius){};
\node[hvertex] (u) at (90+2*\fifth:\uvshift*\radius){};
\draw[hedge] (v0) -- (u) -- (v1);
\draw[hedge] (u) -- (v3);
\draw[hedge] (v0) -- (v) -- (v1);
\draw[hedge] (v) -- (v3);
\node[hvertex,fill=black] (x) at (90+3*\fifth:1*\uvshift*\radius){};
\draw[hedge] (u) -- (x) -- (v);
\node at \captionpos {(b)};
\end{scope}
\end{tikzpicture}
\caption{$x$ in solid black.
%(a) 
%(b) does not contain any forbidden $t$-minor
}
\label{figfig}
\end{figure}

Consider two adjacent vertices $y,z\notin N(C)$, and assume that there is a 
$u\in N(y)\cap N(C)$ that is not adjacent to $z$. 
We may assume that $N(u)\cap C=\{c_1,c_2,c_4\}$ by~\eqref{xnbsclm}.
Then, $zyuc_2c_3$ is an induced $P_5$, which is impossible. Thus:
\begin{equation}\label{commonnbs}
\emtext{
$N(y)\cap N(C)=N(z)\cap N(C)$ for any adjacent $y,z\notin N(C)$. 
}
\end{equation}

Next, fix some vertex $x$ that is not dominated by $C$
(and, by assumption, there is such a vertex). 
As a consequence of~\eqref{commonnbs}, $N(x)\cap N(C)$ separates
$x$ from $C$. In particular, 
\begin{equation}\label{Xsep}
\emtext{
$X:=N(x)\cap N(C)$ is a separator.
}
\end{equation}

Consider two vertices $u,v\in X$. Then, by~\eqref{xnbsclm}, 
each of $u$ and $v$ have exactly three neighbours in $C$, not all of which 
are consecutive. We may assume that $N(u)\cap V(C)=\{c_1,c_2,c_4\}$.

First, assume that $uv\in E(G)$, and suppose that 
the neighbourhoods of $u$ and $v$ in $C$ are the same.  
This, however, is impossible as then $u,v,c_1,c_2$ form a $K_4$.
Therefore, $uv\in E(G)$ implies $N(u)\cap V(C)\neq N(v)\cap V(C)$. 

Now assume $uv\notin E(G)$.
By Lemma~\ref{nbtypestwolem}, there are only two possible configurations (up to 
isomorphy)
for the neighbours of $v$ in $C$; these are (a) and (c) in Figure~\ref{twonbfig}.
The first of these,~(a) in Figure~\ref{figfig}, is impossible, 
as this is a graph of $\mathcal F$; see Figure~\ref{K4fig}~(c).
Thus, we see that $u,v$ are as in~(b) of Figure~\ref{figfig},
that is, that $u$ and $v$ have the same neighbours in $C$. 

To sum up, we have proved that:
\begin{equation}\label{cyclenbs}
uv\in E(G) \Leftrightarrow N(u)\cap V(C)\neq N(v)\cap V(C)
\emtext{ for any two }u,v\in X
\end{equation}

An immediate consequence is that the neighbourhoods in $C$ partition $X$
into stable sets $X_1, \dots, X_k$ such that $X_i$ is complete to $X_j$ whenever $i\neq j$. 
As $X$ cannot contain any triangle---together with $x$ this would result 
in a $K_4$---it follows that $k\leq 2$. 
If $k=1$, we put $X_2=\emptyset$ so that always $X=X_1\cup X_2$. 

We claim that $X$ is a harmonious cutset. 
As $X$ is a separator, by~\eqref{Xsep}, we only need to prove that $(X_1,X_2)$
is $G$-harmonious. 
For this, we have to check the parities
of induced $X_1$-paths and of $X_2$-paths; 
since $X_1$ is complete to $X_2$ any induced $X_1$--$X_2$ path 
is a single edge and has therefore odd length. 

Suppose there is an odd induced $X_1$-path or $X_2$-path. 
Clearly, we may assume there is such a path $P$ that starts in $u\in X_1$ and ends in $v\in X_1$. 
As $X_1$ is stable, and as $G$ is $P_5$-free, it follows that $P$ has length $3$. 
So, let $P=upqv$.

Let us consider the position of $p$ and $q$ relative to $C$.
We observe that neither $p$ nor $q$ can be in $C$. 
Indeed, if, for instance, $p$ was in $C$ 
then $p$ would also be a neighbour of $v$ since $N(u) \cap V(C)= N(v) \cap V(C)$,
by~\eqref{cyclenbs}. This, however, is impossible as $P$ is induced.

Next, assume  that $p,q \notin N(C)$ holds.  
Since $p$ and $q$ are adjacent, we can apply \eqref{commonnbs} to $p$ and $q$,
which results in $N(p)\cap N(C)=N(q)\cap N(C)$. 
However, as $u$ lies in  $N(p)\cap N(C)$ it then also is a neighbour of $q$, 
which contradicts that $upqv$ is induced.

It remains to consider the case when one of $p$ and $q$, $p$ say, lies in $N(C)$.
As $p$ is adjacent to $u$ but not to $v$, both of which lie in $X_1$ and are therefore
non-neighbours, 
it follows from~\eqref{cyclenbs}
that $p\notin X$. In particular, $p$ is not a neighbour of $x$, which means
that $puxv$ is an induced path. 

Suppose there is a neighbour $c\in V(C)$ of $p$ that is not adjacent to $u$. 
By~\eqref{cyclenbs}, $c$ is not adjacent to $v$ either, so that $cpuxv$ forms an induced $P_5$, 
a contradiction. 
Thus, $N(p)\cap V(C) \subseteq N(u)$ has to hold. 
By~\eqref{xnbsclm}, we may assume that the neighbours of $u$ in $C$ are precisely $c_1,c_2,c_4$. 
As $u$ and $p$ are adjacent, $p$ cannot be neighbours with both of $c_1$ and $c_2$, 
as this would result in a $K_4$. Thus, we may assume that $N(p)\cap V(C)=\{c_2,c_4\}$. 
(Note, that $p$ has at least two neighbours in $C$, by Lemma~\ref{nbtypeslem}.)

To conclude, we observe that $pc_4c_5c_1c_2p$ forms a $5$-hole, in which $u$ 
has four neighbours, namely $c_1,c_2,c_4,p$. This, however, 
is in direct contradiction to Lemma~\ref{nbtypeslem}, which means that our assumption
is false, and there is no odd induced $X_1$-path, and no such $X_2$-path either. 
Consequently,  $(X_1,X_2)$ is $G$-harmonious, and $X=X_1\cup X_2$ therefore a harmonious
cutset. 
\end{proof}

\begin{proposition}\label{propfivehole}
Let $G$ be a $t$-imperfect graph. Then either $G$ contains an odd 
hole or it contains $K_4$ or $C_7^2$ as an induced subgraph.
\end{proposition}
\begin{proof}
Assume that $G$ does not contain any odd hole and neither $K_4$ nor
$C_7^2$ as an induced subgraph.
Observe that any odd antihole of length $\geq 9$ contains $K_4$.
Since the complement of a $5$-hole is a $5$-hole, 
and since $C_7^2$ is the odd antihole of length $7$, it follows that $G$ cannot contain any odd antihole at all. 

Now, by the strong perfect graph theorem it follows that $G$ is perfect.
(Note that we do not need the full theorem but only the 
far easier version for $K_4$-free graphs; see Tucker~\cite{Tucker77}.) 
Since $G$ does not contain any $K_4$ it is therefore $t$-perfect as well.
\end{proof}

\begin{lemma}\label{lemmainresnot}
Let $G$ be an $\mathcal F$-free graph.
If $G$ contains a $5$-hole, and if every $5$-hole is dominating 
then $G$ is near-bipartite.  
\end{lemma}

\begin{proof}
Let $G$ contain a $5$-hole, and assume every $5$-hole to be dominating.
Suppose that the lemma is false, i.e.\ that $G$ fails to be 
near-bipartite.
In particular, there is a vertex $v$ such 
that $G-N(v)$ is not bipartite, and therefore contains an induced
odd cycle $T$. As any $5$-hole is dominating and any $k$-hole with $k>5$ contains an induced $P_5$, $T$ has to be a triangle. Let $T=xyz$. 
We distinguish two cases, both of which will lead to a contradiction.

\medskip
\noindent{\bf Case:} $v$ lies in a $5$-hole $C$.\\
Let $C=c_1\ldots c_5c_1$, and
$v=c_1$. Then $T$ could meet $C$ in $0,1$ or $2$ vertices. If $T$ has two vertices
with $C$ in common, these have to be $c_3$ and $c_4$ as the others are neighbours of $v$. 
Then, the third vertex of $T$ has two consecutive neighbours in $C$, which means that by Lemma~\ref{nbtypeslem} its third neighbour in $C$ has to be $c_1=v$, which is impossible. 

Next, suppose that $T$ meets $C$ in one vertex, $c_3=z$, say. 
By Lemma~\ref{nbtypeslem},
each of $x,y$ has to have a neighbour opposite of $c_3$ in $C$, that is, 
either $c_1$ or $c_5$. As $c_1=v$, both of $x,y$ are adjacent with $c_5$. 
The vertices $x,y$ could have a third neighbour in $C$; this would necessarily
be $c_2$. However, not both can be adjacent to $c_2$ as then $x,y,c_2,c_3$ 
would induce a $K_4$. Thus, assume $x$ to have exactly $c_3$ and $c_5$ as neighbours
in $C$. This means that $C'=c_3xc_5c_1c_2c_3$ is a $5$-hole in which $y$ 
has at least three consecutive neighbours, $c_3,x,c_5$, which is impossible
(again, by Lemma~\ref{nbtypeslem}).

Finally, suppose that $T$ is disjoint from $C$. Each of $x,y,z$ has at least two neighbours
among $c_2,\ldots, c_5$, and no two have $c_3$ or $c_4$ as neighbour; otherwise
we would have found a triangle in $G-N(v)$
 meeting $C$ in exactly one vertex, and could reduce
to the previous subcase. Thus, we may assume that $x$ is adjacent to 
$c_2$ and $c_5$. Moreover, since no vertex of $x,y,z$ can be adjacent to both $c_3$ and $c_4$
(as then it would also be adjacent to $c_1$, by Lemma~\ref{nbtypeslem}) and no $c_i \in C$ can be adjacent to all vertices of $T$ (because otherwise $c_i,x,y,z$ would form a $K_4$), 
it follows that we may assume
that $y$ is adjacent to $c_2$ but not to $c_5$, while $z$
is adjacent to $c_5$ but not to $c_2$. 
Then, $c_1c_2yzc_5c_1$ is a $5$-hole in which $x$ has four neighbours,
in obvious contradiction to Lemma~\ref{nbtypeslem}.
Therefore, this case is impossible.

\medskip
\noindent{\bf Case:} $v$ does not lie in any $5$-hole.\\
Let $C=c_1\ldots c_5 c_1$ be a $5$-hole. 
Since every $5$-hole is dominating, $v$ has a neighbour in $C$, and thus 
is, by Lemma~\ref{nbtypeslem}, either as in~(f) of Figure~\ref{nbtypesfig}
or as in~(d). The latter, however, is impossible since then $v$ would 
be contained in a $5$-hole. Therefore, we may assume 
that the neighbours of $v$ in $C$ are precisely $\{ c_1, c_2, c_4\}$. 
As a consequence, $T$ can meet $C$ in at most $c_3$ and $c_5$,
both not in both as $C$ is induced.

Suppose $T=xyz$ meets $C$ in $x=c_3$. 
If $y$ is not adjacent to either of $c_1$ and $c_4$, 
then $c_1 v c_4 xy$ forms an induced $P_5$. 
If, on the other hand, $y$ is adjacent to $c_4$ then, by
Lemma~\ref{nbtypeslem}, also to $c_1$.
Thus, $y$ is either adjacent to $c_1$ or to both $c_1$ and $c_4$. 
The same holds for $z$. Since $y$ and $z$ are adjacent, they cannot both have three neighbours in $C$ (otherwise $G$ would contain a $K_4$). Suppose $N(y)\cap C = \{x, c_1\}$. But then $x c_4 c_5 c_1 y x$ forms an induced $5$-cycle in which $z$ has at least three consecutive neighbours; a contradiction to Lemma~\ref{nbtypeslem}.

Consequently, $T$ is disjoint from every $5$-hole.
By Lemma~\ref{nbtypeslem}, each of $x,y,z$ has neighbours in~$C$ as in~(d) or~(f) of Figure~\ref{nbtypesfig}.
However, if any of $x,y,z$ has only two neighbours in~$C$ as in~(d) then that vertex together
with four vertices of~$C$ forms a $5$-hole that meets~$T$---this is precisely the situation of the
previous subcase. Thus, we may assume that all vertices of $T$ 
have three neighbours in~$C$ as in~(f) of Figure~\ref{nbtypesfig}. 
If we consider the possible configurations of two non-adjacent vertices which have three 
neighbours in~$C$  (namely $v$ and a vertex of $T$) as we have done in Lemma~\ref{twonbfig},
we see that only~(a) and~(c) in Figure~\ref{twonbfig} are possible.
But then each vertex of $T$ has to be adjacent to $c_4$, which means that $T$ together with~$c_4$
induces a $K_4$, which is impossible.
\end{proof}

\begin{proof}[Proof of Lemma~\ref{mainlem}]
Suppose that $G$ is a $t$-imperfect and but $\mathcal F$-free. 
By deleting suitable vertices we may assume that every proper induced subgraph of $G$ is $t$-perfect. 
In particular, by Lemma~\ref{noharmlem}, $G$ does not admit a harmonious cutset.
Since $G$ is $t$-imperfect it contains an odd hole, by Proposition~\ref{propfivehole}, 
and since $G$ is $P_5$-free, the odd hole is of length $5$.
From Lemma~\ref{any5hole} we deduce that any $5$-hole is dominating. 
Lemma~\ref{lemmainresnot} implies that $G$ is near-bipartite. 

Noting that both \aweb{13}{4} and \aweb{19}{7}, as well 
as any M\"obius ladder or any odd wheel larger than $W_5$, 
contain an induced $P_5$, 
we see with Proposition~\ref{propnearbip} that $G$ is $t$-perfect after all.
\end{proof}

By Lemma~\ref{mainlem}, a $P_5$-free graph is either $t$-perfect or contains 
one of eight $t$-imperfect graphs as an induced subgraph. Obviously, 
checking for these forbidden induced subgraphs can be done in 
polynomial time, so that we get as immediate algorithmic consequence:

\begin{theorem}\label{polythm}
$P_5$-free $t$-perfect graphs can be recognised in polynomial time. 
\end{theorem}

We suspect, but cannot currently prove, that $t$-perfection can be 
recognised as well in polynomial time in near-bipartite graphs.

\section{Colouring}\label{colsec}

Can $t$-perfect graphs always be coloured with few colours? This 
is one of the main open questions about $t$-perfect graphs. 
A conjecture by Shepherd and Seb\H o asserts that four colours 
are always enough:

\begin{conjecture}[Shepherd; Seb\H o~\cite{Aperso}]\label{colconj}
Every $t$-perfect graph is $4$-colourable.
\end{conjecture}

The conjecture is known to hold in a number of graph classes, 
for instance in claw-free graphs, where even three colours are already
sufficient; see~\cite{tperfect}. It is straightforward to verify the
conjecture for near-bipartite graphs:

\begin{proposition}
Every   near-bipartite $t$-perfect graph is $4$-colourable.
\end{proposition}
\begin{proof}
Pick any vertex $v$ of a near-bipartite
and $t$-perfect graph $G$. Then $G-N(v)$ is bipartite and may be coloured 
with colours $1,2$. On the other hand, as $G$ is $t$-perfect the neighbourhood
$N(v)$ necessarily induces a bipartite graph as well; otherwise $v$ together 
with a shortest odd cycle in $N(v)$ would form an odd wheel. 
Thus we can colour the vertices in $N(v)$ with the colours $3,4$.
% v bekommt schon automatisch entweder Farbe 1 oder 2, da wir anfangs G-N(v) färben
\end{proof}
Near-bipartite $t$-perfect graphs can, in general, not be coloured with fewer colours. 
Indeed, this is even true if we restrict ourselves further to complements
of line graphs, which is a subclass of near-bipartite graphs.
Two $t$-perfect graphs in this class that need four colours are:
$\overline{L(\Pi)}$,
the complement of the line graph of the prism, and $\overline{L(W_5)}$.
The former was found by Laurent and Seymour (see~\cite[p.~1207]{LexBible}),
while the latter was discovered by Benchetrit~\cite{YohPhD}. 
Moreover, Benchetrit showed that any $4$-chromatic $t$-perfect
complement of a line graph contains one of $\overline{L(\Pi)}$
and $\overline{L(W_5)}$ as an induced subgraph.

How about $P_5$-free $t$-perfect graphs? 
Applying insights of Seb\H o and of Sumner, 
Benchetrit~\cite{YohPhD} proved that $P_5$-free $t$-perfect graphs are $4$-colourable. 
This is not tight:
\begin{theorem}\label{P5colthm}
Every $P_5$-free $t$-perfect graph $G$ is $3$-colourable.
\end{theorem}

For the proof we use that there is a finite number 
of obstructions for $3$-colourability in $P_5$-free graphs:

\begin{theorem}[Maffray and Morel~\cite{MM12}]\label{MMthm}
A $P_5$-free graph is $3$-colourable if and only if it does not
contain $K_4$, $W_5$, $C^2_7$, $\aweb{10}{2}$, $\aweb{13}{3}$
or any of the seven graphs in~Figure~\ref{MMfig} as an induced
subgraph.
\end{theorem}
(Maffray and Morel call these graphs $F_1$--$F_{12}$. The graphs 
$K_4$, $W_5$, $C^2_7$, $\aweb{10}{2}$, $\aweb{13}{3}$ are 
respectively $F_1$, $F_2$, $F_9$, $F_{11}$ and $F_{12}$.) 
A similar result was obtained by Bruce, Ho\`ang and Sawada~\cite{BHS09}, who gave a list 
of five forbidden (not necessarily induced) subgraphs. 

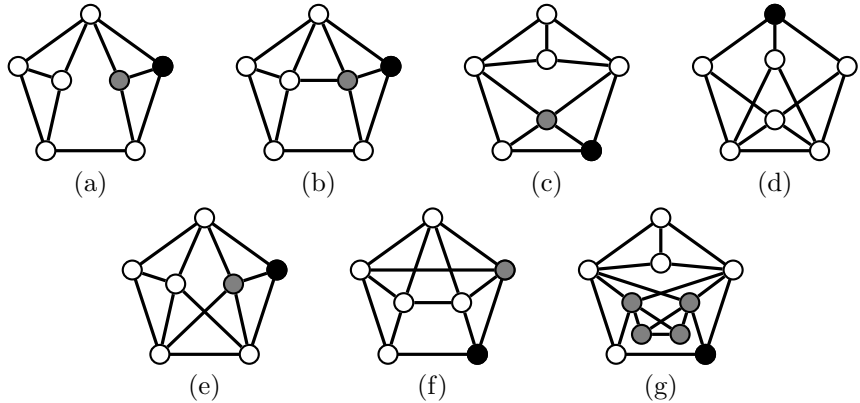
\begin{figure}
\centering
\begin{tikzpicture}

\def\hshift{3cm}
\def\vshift{-2.7cm}
\def\radius{1cm}
\def\captionpos{(0,0.5*\vshift+0.1cm)}
\def\fifth{360/5}

\def\uvshift{0.4}

\tikzstyle{delvx}=[hvertex,fill=gray]
\tikzstyle{convx}=[hvertex,fill=black]

\begin{scope}
\drawfive
 \node[hvertex,convx] (v4) at (72*4+90:\radius){}; 

\node[hvertex,delvx] (v) at (90-\fifth:\uvshift*\radius){};
\node[hvertex] (u) at (90+\fifth:\uvshift*\radius){};
\draw[hedge] (v) edge (v0) edge (v4) edge (v3);
\draw[hedge] (u) edge (v0) edge (v1) edge (v2);
\node at \captionpos {(a)};
\end{scope}

\begin{scope}[shift={(\hshift,0)}]
\drawfive
\node[hvertex,convx] (v4) at (72*4+90:\radius){}; 
 
\node[hvertex,delvx] (v) at (90-\fifth:\uvshift*\radius){};
\node[hvertex] (u) at (90+\fifth:\uvshift*\radius){};
\draw[hedge] (v) edge (v0) edge (v4) edge (v3);
\draw[hedge] (u) edge (v0) edge (v1) edge (v2) edge (v);
\node at \captionpos {(b)};
\end{scope}

\begin{scope}[shift={(2*\hshift,0)}]
\drawfive
\node[hvertex,convx] (v3) at (72*3+90:\radius){}; 
\node[hvertex] (u) at (90:\uvshift*\radius){};
\node[hvertex,delvx] (v) at (270:\uvshift*\radius){};
\draw[hedge] (u) edge (v0) edge (v4) edge (v1);
\draw[hedge] (v) edge (v1) edge (v2) edge (v3) edge (v4);
\node at \captionpos {(c)};
\end{scope}

\begin{scope}[shift={(3*\hshift,0)}]
\drawfive
\node[hvertex,convx] (v0) at (72*0+90:\radius){};
\node[hvertex] (u) at (90:\uvshift*\radius){};
\node[hvertex] (v) at (270:\uvshift*\radius){};
\draw[hedge] (u) edge (v0) edge (v2) edge (v3);
\draw[hedge] (v) edge (v1) edge (v2) edge (v3) edge (v4);
\node at \captionpos {(d)};
\end{scope}

\begin{scope}[shift={(0.5*\hshift,\vshift)}]
\drawfive
\node[hvertex,convx] (v4) at (72*4+90:\radius){};
\node[hvertex,delvx] (v) at (90-\fifth:\uvshift*\radius){};
\node[hvertex] (u) at (90+\fifth:\uvshift*\radius){};
\draw[hedge] (v) edge (v0) edge (v4) edge (v3) edge (v2);
\draw[hedge] (u) edge (v0) edge (v1) edge (v2) edge (v3);
\node at \captionpos {(e)};
\end{scope}

\begin{scope}[shift={(1.5*\hshift,\vshift)}]
\drawfive
\node[hvertex,convx] (v3) at (72*3+90:\radius){};
\node[hvertex,delvx] (v4) at (72*4+90:\radius){};
\node[hvertex] (v) at (90-1.5*\fifth:\uvshift*\radius){};
\node[hvertex] (u) at (90+1.5*\fifth:\uvshift*\radius){};
\draw[hedge] (v) edge (v0) edge (v4) edge (v3);
\draw[hedge] (u) edge (v0) edge (v1) edge (v2);
\draw[hedge] (v1) -- (v4);
\draw[hedge] (u) -- (v);

\node at \captionpos {(f)};
\end{scope}

\begin{scope}[shift={(2.5*\hshift,\vshift)}]
\drawfive
\node[hvertex,convx] (v3) at (72*3+90:\radius){};
\node[hvertex] (u) at (90:\uvshift*\radius){};
\node[hvertex,delvx] (a) at (270-\fifth:\uvshift*\radius){};
\node[hvertex,delvx] (b) at (270+\fifth:\uvshift*\radius){};
\node[hvertex,delvx] (c) at (270+25:0.6*\radius){};
\node[hvertex,delvx] (d) at (270-25:0.6*\radius){};
\draw[hedge] (u) edge (v0) edge (v4) edge (v1);
\draw[hedge] (a) edge (v4) edge (v2) edge (v1);
\draw[hedge] (b) edge (v4) edge (v3) edge (v1);
\draw[hedge] (c) -- (d);
\draw[hedge] (c) edge (a) edge (b);
\draw[hedge] (d) edge (a) edge (b);
\node at \captionpos {(g)};
\end{scope}
\end{tikzpicture}
\caption{The remaining $4$-critical $P_5$-free graphs
of Theorem~\ref{MMthm}; in Maffray and Morel~\cite{MM12} these 
are called $F_3$--$F_8$ and $F_{10}$.
In each graph,
deleting the grey vertices and then $t$-contracting at the black vertex
results in $K_4$.}\label{MMfig}
\end{figure}

\begin{proof}[Proof of Theorem~\ref{P5colthm}]
Any  $P_5$-free graph  $G$ that cannot be coloured with 
three colours contains one of the twelve induced 
subgraphs of Theorem~\ref{MMthm}. 
Of these twelve graphs, we already know that
$K_4$, $W_5$, $C^2_7$, $\aweb{10}{2}$, $\aweb{13}{3}$
are $t$-imperfect, and thus cannot be induced subgraphs of a
$t$-perfect graph. It remains to consider the seven graphs
in Figure~\ref{MMfig}. 
These graphs are $t$-imperfect, too: each can be 
turned into $K_4$ by first deleting the grey vertices and then 
performing a $t$-contraction at the respective black vertex. 
\end{proof}

We mention that Benchetrit~\cite{YohPhD} also showed that 
$P_6$-free $t$-perfect graphs are $4$-colourable.
This is tight: both $\overline{L(\Pi)}$ 
and $\overline{L(W_5)}$ (and indeed all complements of  line graphs)
are $P_6$-free. We do not know whether $P_7$-free $t$-perfect graphs
are $4$-colourable. 

We turn now to fractional colourings. A motivation for Conjecture~\ref{colconj} was certainly the fact that
the \emph{fractional chromatic number} $\chi_f(G)$ of a $t$-perfect graph~$G$
is always bounded by~$3$. 
More precisely, if $\og(G)$ denotes the \emph{odd girth} of $G$, 
that is, the length of the shortest odd cycle, then 
$\chi_f (G) = 2 \frac{\og(G)}{\og(G)-1}$ 
as long as $G$ is $t$-perfect (and non-bipartite). This follows from linear programming 
duality; see for instance Schrijver~\cite[p.~1206]{LexBible}.

Recall that a graph $G$ is perfect if and only if $\chi(H)=\omega(H)$
for every induced subgraph $H$ of $G$. As odd cycles 
seem to play a somewhat similar role for $t$-perfection as
cliques play for perfection, one might conjecture
that $t$-perfection is characterised in an analogous way: 

\begin{conjecture} \label{tpfraccol}
A graph $G$ is $t$-perfect if and only if 
$\chi_f (H) = 2 \frac{\og(H)}{\og(H)-1}$ for every non-bipartite $t$-minor $H$ of $G$.
\end{conjecture}

Note that the conjecture becomes false if, instead of $t$-minors, only induced subgraphs $H$
are considered. Indeed, in the $t$-imperfect graph obtained from $K_4$ by subdividing some edge twice,
all induced subgraphs satisfy the condition (but not the $t$-minor $K_4$).

An alternative but equivalent formulation of the conjecture is: 
 $\chi_f (G) > 2 \frac{\og(G)}{\og(G)-1}$ holds for every minimally $t$-imperfect 
graph $G$. It is straightforward to check that all minimally $t$-imperfect graphs
that are known to date satisfy this. In particular, it follows that the 
conjecture is true for $P_5$-free graphs, for near-bipartite graphs,
as well as for claw-free graphs; see~\cite{tperfect} for the minimally $t$-imperfect
graphs that are claw-free.

\subsection*{Acknowledgment}

We thank Oliver Schaudt for pointing out~\cite{MM12}. 
We also thank one of the referees for bringing  the work of 
Holm et al.~\cite{HTW10} to our attention, as well as 
identifying inaccuracies that led to a (hopefully!) clearer presentation 
of Theorems~\ref{polythm} and~\ref{P5colthm}.

\bibliographystyle{amsplain}
\bibliography{bibtpp5}

\vfill

\small
\vskip2mm plus 1fill
\noindent
Version \today{}
\bigbreak

\noindent
Henning Bruhn
{\tt <henning.bruhn@uni-ulm.de>}\\
Elke Fuchs
{\tt <elke.fuchs@uni-ulm.de>}\\
Institut f\"ur Optimierung und Operations Research\\
Universit\"at Ulm, Ulm\\
Germany\\

\end{document}